\documentclass[british,11pt]{article}

\usepackage[british]{babel}
\usepackage{amsmath,amssymb,enumerate,amsthm}
\usepackage{cite}
\usepackage[latin1]{inputenc}
\usepackage[T1]{fontenc}
\usepackage{relsize}

\usepackage[normalem]{ulem}

\usepackage{psfig,labelfig}
\usepackage{graphicx}
\usepackage[dvips]{epsfig}

\usepackage{lineno} 
\usepackage{color} 
\definecolor{darkred}{rgb}{0.78,0.15,0.10}
\definecolor{darkblue}{rgb}{0.05, 0.15, 0.76}
\definecolor{darkgreen}{rgb}{0.18,0.60,0.12}
\definecolor{darkyellow}{rgb}{0.99,0.67,0.00}

\newtheorem{thm}{Theorem}[section]
\newtheorem{cor}[thm]{Corollary}
\newtheorem{lem}[thm]{Lemma}

\newtheorem{exa}[thm]{Example}

\newcommand{\C}{{\mathbb C}}

\newcommand{\Z}{{\mathbb Z}}

\newcommand{\cG}{{\mathcal G}}
\newcommand{\cT}{{\mathcal T}}
\newcommand{\cH}{{\mathcal H}}

\DeclareMathOperator{\arcsinh}{arcsinh}
\DeclareMathOperator{\arccosh}{arccosh}

\DeclareMathOperator{\area}{area}

\DeclareMathOperator{\sy}{sys}

\title{Short homology bases for hyperelliptic hyperbolic surfaces }

\author{Peter Buser, Eran Makover and Bjoern Muetzel}

\begin{document}
\maketitle

\begin{abstract}
Given a hyperelliptic hyperbolic surface $S$ of genus $g \geq 2$, we find bounds on the lengths of homologically independent loops on $S$. As a consequence, we show that for any $\lambda \in (0,1)$ there exists a constant $N(\lambda)$ such that every such surface has at least $\lceil \lambda \cdot \frac{2}{3} g \rceil$ homologically independent loops of length at most $N(\lambda)$, extending the result in \cite{mu1} and \cite{bps}. This allows us to extend the constant upper bound obtained in \cite{mu1} on the minimal length of non-zero period lattice vectors of hyperelliptic Riemann surfaces to almost $\frac{2}{3} g$ linearly independent vectors.
\end{abstract}

\tableofcontents

\section{Introduction}
The most prominent short curve on a (Riemannian) surface is the \textit{systole}, which is a shortest non-contractible closed geodesic on the surface. By abuse of notation we shall also call its length $\sy(\cdot)$ the systole. Similarly the \textit{homology systole} $\sy_h(\cdot)$ is the length of a shortest non-separating simple closed geodesic. In this article we are interested in short curves on hyperelliptic Riemann surfaces, where a Riemann surface denotes a compact hyperbolic surface without boundary of genus $g \geq 2$. Questions about sets of short simple closed curves on Riemann surfaces date back to Bers \cite{be}. Every such surface can be decomposed into pairs of pants, i.e. into three-holed spheres, by cutting it  along $3g - 3$ simple closed non-intersecting geodesic curves. These curves can always be chosen in such a way that the upper bound of their hyperbolic lengths is of order $g$ (\cite{bse0}, \cite{pa1}). The best such upper bound is called the Bers' constant. The importance of these curves lies in the fact that they can be used to describe a rough fundamental domain for the moduli space $\mathcal{M}_g$ of Riemann surfaces of genus $g$ (see \cite[chapter 6.6]{bu}).\\
The same question can be asked about homology bases of Riemann surfaces. Here one would like to estimate the lengths of closed geodesic loops constituting a basis for the homology of a given genus $g \geq 2$ Riemann surface. This is related to the Schottky problem, which asks to characterize the sublocus of Jacobian varieties of Riemann surfaces of genus $g \geq 2$ among the principally polarized Abelian varieties of dimension $g$. The length estimates of the homology basis can be used to give such a characterization. As the mapping $t$ that assigns to a Riemann surface its Jacobian variety is injective, this is a problem which is similar to finding a rough fundamental domain of $\mathcal{M}_g$. Here we obtain a coarse domain in which $t(\mathcal{M}_g)$ lies. This study was initiated in \cite{bs} and extended in \cite{mu1}. Bounds for the whole basis were then given in \cite{bps}:

\begin{thm}[Balacheff-Parlier-Sabourau] Let $S$ be a hyperbolic surface of genus $g \geq 2$ with  homology systole $\sy_h(S)$. Then there exist $2g$ loops $\alpha_1, \ldots,\alpha_{2g}$ which induce a basis of $H_1(S,\Z)$ such that 
\[
\ell(\alpha_k) \leq C_0 \frac{ \log(2g-k+2)}{2g-k+1} \cdot g, \text{ \ where \ } C_0 = \frac{2^{16}}{min\{\sy_h(S),1\}}.
\]
Especially the median length $(k = g)$ is bounded by $C_0 \log(g+2)$ and the $(\alpha_k)_k$ are bounded by $C_0\cdot g$. 
\label{thm:bps}
\end{thm}
Given a lower bound on the homology systole of $S$, bounds of this order extend directly to the squared lengths of almost $g$ linearly independent vectors in the lattice of the Jacobian of $S$. The length bound of this result for the Jacobian was improved in \cite{hw}. Our main goal in this article is to give a similar characterization for hyperelliptic Riemann surfaces. We show that in terms of these geometric invariants an even lower, constant bound can be found for a large portion of a suitable homology basis. 

\begin{thm} Let $S$ be a hyperelliptic Riemann surface of genus $g \geq 2$. Then there exist $\lceil  \frac{2g+2}{3}  \rceil$ geodesic loops  $(\alpha_k)_{k=1,\ldots,\left \lceil \frac{2g+2}{3}\right \rceil}$ that can be extended to a homology basis of $H_1(S,\Z)$ such that
\begin{equation*}
\ell(\alpha_{k}) \leq  4\log \left(\frac{ 12(g-1) }{ 2g+5-3k} +2 \right) \text{ for all } k=1, \dots, \left\lceil \frac{2g+2}{3}\right \rceil.
\end{equation*}
\label{thm:hyper_length_intro}
\end{thm}

Summarizing this result for a fraction of the homology basis, we have:

\begin{cor} Let $S$ be a hyperelliptic Riemann surface of genus $g \geq 2$. Then for any $\lambda \in (0,1)$ there exist $\lceil \lambda \cdot \frac{2}{3} g \rceil$ geodesic loops  $(\alpha_k)_{k=1,\ldots,\left \lceil \lambda \cdot \frac{2}{3}g \right \rceil}$, that can be extended to a homology basis of $H_1(S,\Z)$ such that 
\[
\ell(\alpha_k) \leq N(\lambda): =  4\log \left( \frac{ 6}{ 1-\lambda} +2 \right)   \text{ for all } k \in \{1 ,\ldots, \left \lceil   \lambda \cdot \frac{2}{3}g \right \rceil \}.
\]
\label{thm:hyper_lengths_cor_intro}
\end{cor}
Again this result can be extended to a set of vectors of a lattice basis of the Jacobian $J(S)$ of a hyperelliptic surface.

\begin{thm} Let $S$ be a hyperelliptic hyperbolic surface of genus $g \geq 2$. Then for any $\lambda \in (0,1)$ there exist $\lceil \lambda \cdot \frac{2}{3} g \rceil$ linearly independent vectors $(v_k)_{k=1,\ldots,\left \lceil \lambda \cdot \frac{2}{3}g \right \rceil}$  in the lattice of the Jacobian torus $J(S)$ that can be extended to a lattice basis, such that
\[
  \|v_k\|^2 = E(\sigma_k) \leq D(\lambda) := \frac{N(\lambda)}{\pi-2 \cdot \arcsin(\frac{1}{\cosh(w(\lambda)})}   \text{ \ for all \ } k \in \{1 ,\ldots, \left \lceil \lambda \cdot \frac{2}{3}g \right \rceil \}.
\]
\label{thm:hyper_Jac_intro}
\end{thm}
Note that this result gives again an upper bound $D(\lambda)$ which is a constant that depends only on $\lambda$. Concerning short curves on hyperelliptic hyperbolic surfaces, results about the decomposition into three-holed spheres have been obtained in \cite{bp}. There the authors show that for hyperelliptic surfaces of genus $g$, Bers' constant is at most of order $\sqrt{g}$. Concerning the systole and homology systole of hyperelliptic surfaces results have been obtained in \cite{ba2} and \cite{mu1}. In \cite{ba2} Bavard gives a constant upper bound for the systole of a hyperelliptic Riemann surface which is independent of the genus. This bound is sharp in genus $2$ and $5$. In \cite{mu1} it is shown that the same upper bound holds for the homology systole and the first inequality for the length of a non-zero lattice vector of the Jacobian in Theorem \ref{thm:hyper_Jac_intro}. The theorem restricts the sublocus of Jacobians of hyperelliptic surfaces among the principally polarized Abelian varieties even further. 

The methods used in this article are a continuation of the approach in \cite{mu1}. Let 

\begin{equation*}
\Sigma = S\backslash \phi
\end{equation*}

be the quotient sphere obtained by taking the quotient of the hyperelliptic surface $S$ with respect to the hyperelliptic involution $\phi$. In a first step we let disks grow on the quotient sphere to get upper bounds for curves using an area argument. In a second step we examine our result more closely and remove unwanted curves, to obtain a set of homologically independent curves and prove the length estimates. Finally we show how these estimates can be used to get bounds of the same order for the vectors of a lattice basis of the Jacobian of a hyperelliptic hyperbolic surface. 

\section{Short loops on hyperelliptic Riemann surfaces} \label{sec:loops}

Let $S$ be a hyperelliptic surface of genus $g \geq 2$ with hyperelliptic involution $\phi : S \to S$ and let $p_1^*, \dots, p_{2g+2}^*$ be its fixed points. The quotient $\Sigma$ is a topological sphere with $2g+2$ cone points of cone angle $\pi$ whose vertices $\{p_i\}_{i=1,\ldots,2g+2}$ are the images of the fixed points under the natural projection $\Pi : S\to \Sigma$ (see \textit{Figure \ref{fig:hyp})}. In complex analysis the fixed points of $\phi$ are known to be the Weierstrass points of $S$. We adopt this terminology for the fixed points $p_i^*$ on $S$ as well as the cone points $p_i$ on $\Sigma$.

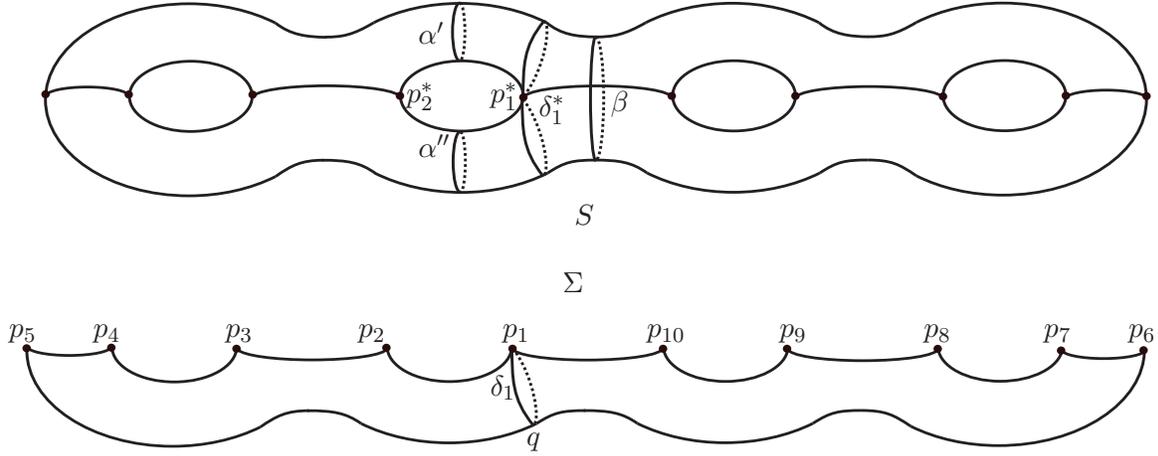
\begin{figure}[!ht]
\SetLabels
\L(.43*.25) $p_1$\\
\L(.42*.77) $p_1^*$\\
\L(.31*.25) $p_2$\\
\L(.35*.77) $p_2^*$\\
\L(.55*.25) $p_{10}$\\
\L(.66*.25) $p_9$\\
\L(.78*.25) $p_8$\\
\L(.88*.25) $p_7$\\
\L(.95*.25) $p_6$\\
\L(.02*.25) $p_5$\\
\L(.09*.25) $p_4$\\
\L(.20*.25) $p_3$\\
\L(.52*.75) $\beta$\\
\L(.36*.90) $\alpha'$\\
\L(.36*.65) $\alpha''$\\
\L(.42*.12) $\delta_1$\\
\L(.45*.01) $q$\\
\L(.46*.74) $\delta^*_1$\\
\L(.49*.50) $S$\\
\L(.48*.35) $\Sigma$\\
\endSetLabels
\AffixLabels{%
\centerline{%
\includegraphics[height=6cm,width=15cm]{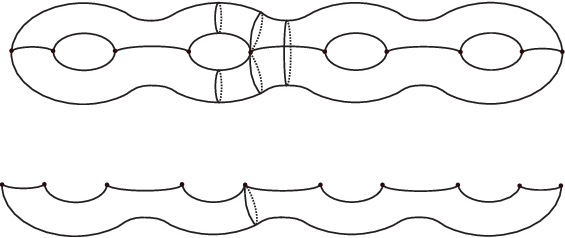}}}

\caption{A hyperelliptic surface $S$ of genus four with ten Weierstrass points.}
\label{fig:hyp}
\end{figure}

In the following process we will first construct a collection of short curves on $\Sigma$. Under $\Pi$ these curves will lift to short loops in $S$. As a result we obtain upper bounds on a certain number $M$ of short non-separating simple closed geodesics on $S$, where $g+1 \leq M \leq 2g+2$.  The number $M$ itself will depend on how the process evolves.

Let $B_r(p_i)$ denote a metric disk of radius $r$ around $p_i$. As long as the $\left( B_r(p_i) \right)_{i=1,\ldots,2g+2} $ are topological disks and mutually disjoint 
their areas equal half the area of a disk of radius $r$ in the hyperbolic plane
\[
\area ( B_r(p_i)) = \pi(\cosh(r)-1),
\] 
and the sum of these areas does not exceed the area of $\Sigma$. 

The idea is to expand the radii of these disks successively until they touch i.e. fail to be embedded. This will give us in each step a curve between Weierstrass points together with an upper bound on the length based on the area of the respective disks. We distinguish the following two cases that mark the interruption of the growing process at the end of each step: two different disks touch each other or a single disk self-intersects. In the first case, we obtain a geodesic arc between two Weierstrass points, which we will call an \textit{edge}, in the second case we obtain a geodesic loop which we will simply call a \textit{loop}. Furthermore, edges and loops will be called \textit{arcs}. We chose this terminology as we will later construct a graph out of these arcs. If several disks happen to come to touch at the same time, we can choose an arbitrary order to turn such a case into a succession of new edges and loops.\\

We start with the first step of our algorithm as follows:\\ 
\\
\textbf{Step 1:} We expand the radii of the $2g+2$ disks simultaneously until one of the following two cases occurs:\\
\begin{itemize}
\item[Case i)] \textit{ the closure of a single disk with radius $r_1$ comes to self-intersect in some point $q$.}\\
\\
We assume without loss of generality that this happens to the closure of $B_{r_1}(p_1)$ (see \textit{Figure \ref{fig:hyp}}). We connect $p_1$ with $q$ with two geodesic arcs of length $r_1$ that meet in $q$ at an angle $\pi$. The two arcs together form a geodesic loop $\delta_1$ (see \textit{Figure \ref{fig:hyp}}). It was shown in \cite{mu1} that $\delta_1$ lifts  to a figure eight closed geodesic $\delta^*_1$ of length $2\ell(\delta_1) = 4r_1$. Splitting the figure eight geodesic at the lift $p^*_1$ into two loops there are two non-separating simple closed geodesics $\alpha'$ and $\alpha''$ in the free homotopy classes of these two loops, where 
\[ 
    \ell(\alpha') = \ell(\alpha'') < \ell(\delta_1) = 2r_1. 
\]
Note that the figure eight geodesic can also be split in another way to get a simple closed geodesic $\beta$.\\

\item[Case ii)] \textit{ the closures of two different disks with radius $r_1$ come to intersect.}\\
\\
In this case we assume that the closures $\overline{B_{r_1}(p_1)}$ and $\overline{B_{r_1}(p_2)}$ are intersecting in a point $q$. We connect $q$ with geodesic arcs of length $r_1$ with $p_1$ and $p_2$. Again we call the union of these two arcs $\delta_1$. For this case it was shown in \cite{mu1} that $\delta_1$ lifts to a non-separating simple closed geodesic $\alpha$ of length 
\[
 \ell(\alpha) = 2 \ell(\delta_1) = 4r_1.
\]
\end{itemize}
In both Case i) and Case ii) we have that
\[
\area(\biguplus_{i=1}^{2g+2}  B_{r_1}(p_i)) = \sum_{i=1}^{2g+2} \area(B_{r_1}(p_i)) = (2g+2) \pi  (\cosh(r_1)-1).
\]
As the union of the disks $\biguplus_{i=1}^{2g+2} B_{r_1}(p_i)$ is embedded in $ \Sigma $, we have furthermore 
\[
\area(\biguplus_{i=1}^{2g+2} B_{r_1}(p_i)) \leq \area(\Sigma) = 2\pi(g-1). 
\]
As $\frac{g-1}{g+1} < 1$ we conclude, that $r_1 < \arccosh(2)$. As $\ell(\delta_1) \leq 2r_1$ we obtain from the above inequality an upper bound for $\ell(\delta_1)$:
\[
\ell(\delta_1) \leq  2\arccosh(2).
\]
Now, if we are in Case i) we set $\alpha_1 := \alpha'$ and in Case ii) we set $\alpha_1 := \alpha$. In both cases $\alpha_1$ is a non-separating closed geodesic and therefore homologically non-trivial. Hence we obtain for the homology systole $\sy_h(S)$: 
\[ 
 \sy_h(S)  \leq  \ell(\alpha_1) \leq 2 \ell(\delta_1) \leq 4 r_1 \leq  4\arccosh(2) = 5.2678...
\]
It is noteworthy that by a refinement of this area estimate Bavard obtains a better upper bound in \cite{ba2}, which is
\begin{equation}
 \sy_h(S) \leq 4\arccosh \left( {\left(
2\sin \left( {\tfrac{{\pi (g + 1)}}
{{12g}}} \right)\right)}^{-1} \right) < 2\log(3+2\sqrt{3}+2\sqrt{5+3\sqrt{3}})=5.1067...
\label{eq:bav}
\end{equation}
We denote by $k_1 \in \{1,2\}$ the number of disks (or Weierstrass points) used to obtain the geodesic arc $\delta_1$.\\ 
\\
\textbf{Step 2:} We now expand the  $2g +2 -k_1$ remaining disks $(B_{r_1}(p_i))_{i= k_1 +1 ,\ldots, 2g+2}$ until 
\begin{itemize}
\item[Case i)] \textit{ the closure of a single disk self-intersects.}\\
\\
In this case let without loss of generality $\overline{B_{r_2}(p_{k_1 +1})}$ be the disk that self-intersects at radius $r_2$. As in Case i) of \textbf{Step 1},  the arc $\delta_2$ connecting $p_{k_1+1}$ through the intersection point of the two disks lifts to a figure eight geodesic, and splitting this geodesic we get a non-separating simple closed geodesic $\alpha_2$ of length $\ell(\alpha_2) \leq 2r_2$.\\ 

\item[Case ii)] \textit{two different disks, either both with radius $r_2$, or one with radius $r_1$ and the other with radius $r_2$ intersect.}\\
\\
Here we assume that  $\overline{B_{r_2}(p_{k_1+1})}$ and  $\overline{B_{r_2}(p_{k_1+2})}$ intersect in the first case respectively, $\overline{B_{r_2}(p_{k_1+1})}$ and  $\overline{B_{r_1}(p_1)}$ in the second case. 
\end{itemize}
In both Case i) or Case ii), we connect the respective Weierstrass points with a geodesic arc $\delta_2 \neq \delta_1$ of length $\ell(\delta_2) \leq2r_2$. As in \textbf{Step 1}, we obtain an upper bound on $r_2 \geq r_1$, as all disks are embedded. In the worst case we get the estimate 
\[
 \sum_{i=k_1+1}^{2g+2} \area(B_{r_2}(p_i)) \leq \area(\Sigma).
\]  
As $\sum_{i=k_1+1}^{2g+2} \area(B_{r_2}(p_i)) = (2g +2 -k_1) \pi \cdot (\cosh(r_2)-1)$ we get that
\[
(2g +2 -k_1) \cdot \pi \cdot (\cosh(r_2)-1)  \leq   \area(\Sigma)  \text{ \ or \ }                  
      r_2 \leq \arccosh\left(\frac{ 2(g-1) }{2g +2 -k_1} +1 \right)
\]
Again let $k_2 \in \{1,2\}$ be the number of disks (or Weierstrass points) used to obtain the geodesic arc $\delta_2$. 

We proceed in this way by expanding in each step the remaining disks further. The number of disks expanding in \textbf{Step \!$\boldsymbol{m}$} is  $2g+2-\sum_{i=1}^{m-1} k_i$. We expand these disks until a single disk self-intersects, in which case we ``consume'' another Weierstrass point or until two different disks intersect in which case we consume two Weierstrass points. In each step we obtain a new geodesic arc connecting one or two Weierstrass points together with an upper bound of its length. The $m$-th step is then:\\
\\
\textbf{Step \!$\boldsymbol{m}$:}  We obtain a geodesic arc $\delta_m$ of length $\ell(\delta_m) \leq 2r_m$ by connecting the respective Weierstrass point or points. As in \textbf{Step 2}, we obtain an upper bound on $r_m$. In the worst case we get

\[
   \sum_{i=j_m+1}^{2g+2} \area(B_{r_m}(p_i)) \leq \area(\Sigma), \text{ \ where \ } j_m := \sum_{i=1}^{m-1} k_i .
\]  
Here $j_m$ is the number of disks consumed in the steps preceding \textbf{Step m}. In this step we get
\begin{eqnarray}
\nonumber
 (2g+2-j_m) \cdot \pi \cdot (\cosh(r_m)-1)  \leq   \area(\Sigma)  \text{ \ or \ }   \\               
 \frac{ \ell(\delta_m) }{2}  = r_m \leq \arccosh\left(\frac{ 2(g-1) }{ 2g+2-j_m} +1 \right).
\label{eq:deltam}
\end{eqnarray}
Each $\delta_m$ obtained in the process leads to a simple closed non-separating geodesic $\alpha_m$ on $S$: if the lift of $\delta_m$ in $S$ is simple closed going through two Weierstrass points we take $\alpha_m$ to be this lift, otherwise the lift is a figure eight geodesic going through exactly one Weierstrass point and we choose $\alpha_m$ to be homotopic to one of the two loops (we cannot use both loops because the pair is separating). In the first case $\alpha_m$ has length $2\ell(\delta_m)$ in the second case the length is smaller than $\ell(\delta_m)$. Simplifying $\arccosh(x) \leq \log(2x)$ and using \eqref{eq:deltam} we obtain the upper bound for the length

\begin{equation}
\ell(\alpha_m) \leq 4\log \left(\frac{ 4(g-1) }{ 2g+2-j_m} +2 \right).
\label{eq:alpham}
\end{equation}

\textbf{Termination:} We halt the process once all disks are expanded in Step $M$, where $g+1 \leq M \leq 2g+2$. As the estimate in Equation \eqref{eq:alpham} uses the number $j_m$ of disks consumed up to \textbf{Step m-1}, this equation is then valid for $m=1, \dots, M $ and we have $j_M \in \{2g, 2g+1\}$.\\

We summarize the result of this section in the following lemma: 
\begin{lem} Let $S$ be a hyperelliptic hyperbolic surface. Let $k_i \in \{1,2\}$ be the number of Weierstrass points used in \textbf{Step m} of our process and $j_m = \sum_{i=1}^{m-1}k_i$. Then there exist $g+1 \leq M \leq 2g+2$ non-separating loops $(\alpha_m)_{m=1,\ldots,M}$ on $S$, such that 
\[
\ell(\alpha_m)  \leq 4\log \left(\frac{ 4(g-1) }{ 2g+2-j_m} +2 \right), \text{ for all } m \in \{1 ,\ldots, M\}
\]      
Especially the $(\alpha_m)_m$ are bounded by $4 \log(4g)$. 
\label{thm:alpham}
\end{lem} 

\textbf{Note:} It is possible that some of these loops in this sequence coincide. We will deal with this problem in detail in the next sections.   

\section{Homology of edges and loops} \label{sec:hom}

Unfortunately, the $(\alpha_m)_m$ obtained in the preceding section are not pairwise distinct in general and, furthermore, their union may be disconnecting $S$. We must therefore find a way to select as many homologically independent cycles among them as possible. For this we make some topological considerations.

\subsection{The graph in $\Sigma$} \label{sub:graph}

We begin with the configuration of the loops and edges on $\Sigma$ resulting from the procedure in Section {\ref{sec:loops}. They form a graph $\cG$ whose vertices are the $2g+2$ cone points of $\Sigma$. For the ease of exposition we use the following terminology for connected components of graphs, slightly deviating from the usual one. 
\begin{itemize}
\item[--] \emph{edge}: an edge in the classical sense that connects two distinct vertices.
\item[--] \emph{loop}: an edge in the classical sense that connects a vertex to itself.
\item[--] \emph{tree}: a tree in the classical sense that has at least one edge.
\item[--] \emph{looped tree} a tree with at least one edge to which exactly one loop is attached.
\end{itemize}

In this terminology, edges are trees but loops are not looped trees.

\begin{figure}[!ht]
\SetLabels
\L(.11*.05) edge\\
\L(.33*.05) loop\\
\L(.57*.05) tree\\
\L(.75*.05) looped tree\\
\endSetLabels
\AffixLabels{%
\centerline{%
\includegraphics[width=12cm]{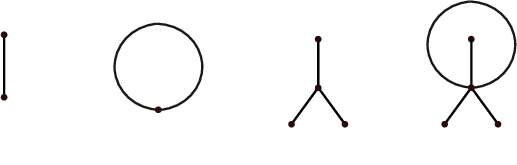}}}
\caption{Terminology for connected components.}
\label{fig:graphcomponents}
\end{figure}
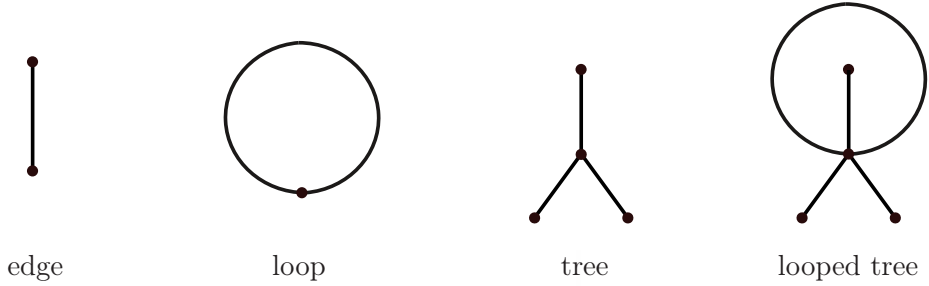

\begin{lem} Each connected component $\cT$ of $\cG$ is either a loop or a tree or a looped tree.\label{lem:graph_nooroneloop}
\end{lem}
\begin{proof} This is seen by induction on the procedure that yields $\cG$. During the induction process we have, in addition,  components that are isolated vertices. These disappear at the end.

At the beginning the components are the cone points. Now assume the lemma, plus isolated vertices, holds for the graph $\cG_{m-1}$ that is present at the beginning of Step $m$. Let $\delta$ be the new geodesic arc obtained at this step. By construction the initial point, $p$, of $\delta$ is a cone point to which no edge or loop of $\cG_{m-1}$ is attached. For the endpoint there are then two possibilities: 1.) it coincides with $p$ in which case $\delta$ together with $p$  is a component of $\cG_{m}$, or 2.) $\delta$ connects to a component $\cT$ of $\cG_{m-1}$ in which case $\cT' := \cT + \delta + p$ is a tree if $\cT$ is a tree, respectively a looped tree  if $\cT$ is a loop or a looped tree. 
The procedure stops at the moment where no more isolated vertices are left.
\end{proof}

\subsection{The lifts in $S$} \label{sub:lifz}
In the following we denote by $X^*$ the lift, i.e. inverse image in $S$ of any subset $X$ of $\Sigma$ under the natural projection $\Pi: S \to \Sigma$.

For each edge $e$ of $\cG$ the lift $e^*$ is a simple closed geodesic, and for each loop $\lambda$ the lift is a closed geodesic in the shape of a figure eight. In Section \ref{sec:prune} we shall remove certain parts from $\cG^*$ so as to turn the remainder into a non-separating set. As a preparation for this we look at the following situation.

Let $\cH$ be obtained from $\cG$ by removing from it a certain number of edges and loops, but keeping all the vertices. Then the connected components of $\cH$ are either isolated vertices, trees, loops or looped trees.  Lift $\cH$ to $\cH^*$ on $S$. We modify $\cH^*$ by removing from each figure eight curve $\lambda^*$ occurring in $\cH^*$ one of the two simple loops it consists of where, for our purpose, it does not matter which of the two is chosen. The remaining loop, $\lambda^\#$, then projects one-to-one onto the loop $\lambda$ under the natural projection $\Pi : S \to \Sigma$. We denote by $\cH^\#$ the lift of $\cH$ modified in this way. If $\cH$ consist of $n$ edges and loops then $\cH^\#$ consists of $n$ simple closed curves.

In the following surface topological considerations we argue with arcs and (simple closed) Jordan curves on the complement $\Sigma \smallsetminus \cH$; they have nothing to do with the graph paths on $\cH$.

\begin{lem}\label{separating} $\cH^\#$ in $S$ is non-separating if and only if any open connected component of $\Sigma \smallsetminus \cH$ contains a Jordan curve $\Gamma$ that separates $p_1, \dots, p_{2g+2}$ into two odd subsets i.e. the number of Weierstrass points on either side of $\Gamma$ is odd.
\label{lem:lift_nonsep}
\end{lem} 

\begin{proof}
We use the following topological models for $\Sigma$ and $S$. For $\Sigma$ we take the compactified complex plane $\overline{\C} = \C \cup \{\infty\}$ and mark it with $2g+2$ distinct points $p_1 = \infty, p_2, \dots, p_{2g+2}$. These play the role of the Weierstrass points. 

To define the twofold branched covering branching over $p_1, \dots, p_{2g+2}$ we use winding numbers. For this we select a base point $p_*$ in  $\C \smallsetminus \{p_2,p_3,\ldots,p_{2g+2} \}$. For any piecewise smooth closed curve $\eta$ in $\C \smallsetminus \{p_2,p_3,\ldots,p_{2g+2} \}$ with initial and endpoint $p_*$ and any $p_k$, $k=2, \dots, 2g+2$, the \emph{winding number}, e.g. \cite{ah} is defined as
\[
W(\eta,p_k) = \frac{1}{2 \pi i}\int_{\eta}\frac{d z}{z - p_k}.
\]
Informally, $W(\eta, p_k)$ is the number of times $\eta$ circles around $p_k$. The \emph{total winding number} of $\eta$ is defined as the sum
\[
W(\eta) = \sum_{k=2}^{2g+2}W(\eta,p_k).
\]
If  now $p \in \C \smallsetminus \{p_2, \dots, p_{2g+2}\}$ is any point and $\delta, \gamma$ are piecewise smooth arcs from $p_*$ to $p$ we shall say that $\delta$ and $\gamma$ are \emph{equivalent} if the total winding number of $\delta \gamma^{-1}$, i.e. $\delta$ followed by $\gamma$ in the reversed sense, is even. We denote the corresponding equivalence class of $\delta$ by $[\delta]$:
\[
 [\delta] = [\gamma]    \Leftrightarrow W(\delta \gamma^{-1}) \text{ \ is even.}  
\]
There are exactly two equivalence classes. \textit{Figure \ref{fig:eq_classes_curv1}} illustrates this for $g=2$ with $[\gamma'] = [\delta]$ and $[\gamma''] \neq [\delta]$.

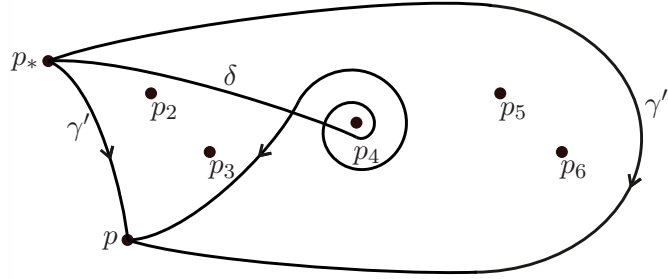
\begin{figure}[!ht]
\centering
\SetLabels
\L(.22*.76) $\,p_*$\\
\L(.27*.52) $\gamma'$\\
\L(.30*.15) $p$\\
\L(.75*.60) $\gamma''$\\
\L(.40*.69) $\delta$\\
\L(.34*.59) $p_2$\\
\L(.38*.39) $\,\,p_3$\\
\L(.50*.44) $\,\,p_4$\\
\L(.63*.59) $p_5$\\
\L(.68*.39) $p_6$\\
\endSetLabels
\AffixLabels{%
\centerline{%
\includegraphics[width=9cm]{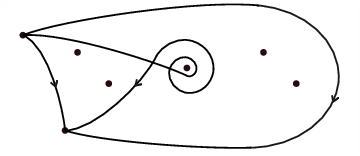}}}
\caption{Equivalence classes of curves on $\C \smallsetminus \{p_2,p_3,\ldots,p_{6} \}$; the pairs $(p,[\gamma'])$ and $(p,[\gamma''])$ are the two covering points of $p$. }
\label{fig:eq_classes_curv1}
\end{figure} 
We now let $\mathcal{S}'$ be the set of all pairs $(p,[\delta])$, where $[\delta]$ is the equivalence class of a curve from $p_*$ to $p$ in $\C \smallsetminus \{p_2, \dots, p_{2g+2}\}$. There is a unique topology on $\mathcal{S}'$ such that the mapping \\ $(p,[\delta]) \to \Pi(p,[\delta]) = p$ is a local homeomorphism from $\mathcal{S}'$ to 
$\C \smallsetminus \{p_2,p_3,\ldots,p_{2g+2} \}$. Furthermore, there is a natural $2g+2$ point compactification $\mathcal{S}$ of $\mathcal{S}'$ such that $\Pi$ extends to a branched covering $\Pi : \mathcal{S} \to \overline{\C}$. We may now identify $\mathcal{S}$ with $S$ and $\overline{\C}$ with $\Sigma$.

Finally, for any closed curve $c : [0,1] \to \Sigma \smallsetminus \{p_1, p_2, \dots, p_{2g+2}\}$ with $c(0) = c(1)$ an arbitrary point in $\Sigma \smallsetminus \{p_1, p_2, \dots, p_{2g+2}\}$ we define the total winding number as
\[
W(c) = W(\gamma c \gamma^{-1}),
\]
where $\gamma$ is an arc from $p_*$ to $c(0)$ in $\Sigma \smallsetminus \{p_1, p_2, \dots, p_{2g+2}\}$. If $c$ is primitive i.e. the parametrized curve runs once along its trace and if $\widehat{c} : [0,1] \to S$ is a  lift  of $c$ satisfying $\Pi \circ \widehat{c} = c$  then
\begin{equation}
\textit{$\widehat{c}$ connects distinct lifts of $c(0)$ if and only if $W(c)$ is odd.}
\label{eq:odd_winding}  
\end{equation}



We first translate this into a statement for open components. Let $\Omega$ be an open connected component of $\Sigma \smallsetminus \cH$. Its lift $\Omega^*$ in $S$ is either connected or consists of two open connected components. If it is connected then there exists a simple arc $\Gamma^* : [0,1] \to \Omega^*$ that connects distinct lifts of the same point $\Pi(\Gamma^*(0)) =  \Pi(\Gamma^*(1))$ in $\Omega$. Restricting $\Gamma^*$ to a sub arc, if necessary, we may assume that, furthermore, its projection $\Gamma = \Pi \circ \Gamma^* : [0,1] \to \Sigma$ is a simple closed curve. By \eqref{eq:odd_winding} $\Gamma$ has an odd winding number. Conversely, suppose that $\Omega$ contains a Jordan curve $\Gamma : [0,1] \to \Omega$ with an odd winding number. Then by \eqref{eq:odd_winding} again there exists a lift $\Gamma^*$ of $\Gamma$ in $S$ connecting distinct lifts of $\Gamma(0)$ in $S$, from which we conclude that $\Omega^*$ must be connected. Altogether we have the following consequence of \eqref{eq:odd_winding} for any connected component $\Omega$ of $\Sigma \smallsetminus \cH$ and its lift $\Omega^*$ in $S$,
\begin{equation}
\textit{$\Omega^*$ is connected if and only if $\Omega$ contains a Jordan curve with an odd winding number.}
\label{eq:odd_winding2}  
\end{equation}

We now turn to the proof of Lemma \ref{lem:lift_nonsep}
Assume first for the `if' direction of the lemma that each component $\Omega$ contains  a Jordan curve with odd winding number. We must show that for any pair of points $p^*,q^* \in S \smallsetminus \cH^\#$ there exists a connecting arc that does not intersect $\cH^\#$.

First of all, there exists an arc $A : [0,1] \to \Sigma$ from $A(0) =p:=\Pi(p^*)$ to $A(1) = q:=\Pi(q^*)$ that avoids the Weierstrass points and furthermore intersects $\cH$ only on its loops. This is so because by Lemma \ref{lem:graph_nooroneloop} the connected components of $\cH$ are just edges, trees, loops or looped trees. This curve we lift to a curve $A^* : [0,1] \to S$ satisfying $\Pi \circ A^*=A $. There are two such lifts. We take the one whose initial point is $A^*(0) = p^*$. Its endpoint $A^*(1)$ is then either $q^*$ or the second lift $q'$ of $q$. In the latter case we modify $A$ on $\Sigma$ by adding at its endpoint $q$ a Jordan curve $\Gamma_q$ with an odd winding number in the open component $\Omega_q$ of $\Sigma \smallsetminus \cH$ that carries $q$. With $A$ modified in this way its lift now satsfies $A^*(0)=p^*$ and $A^*(1)=q^*$.

Now $A^*$ possibly intersects $\cH^\#$, but this can happen only on its loops. We thus further modify $A^*$ as follows. Assume $\alpha^*$ is a small arc on $A^*$ with the ``unfortunate'' property that it intersects one of the lifted loops $\lambda^\#$ that has been kept for the definition of $\cH^\#$. Then $\alpha := \Pi \circ \alpha^*$ intersects the corresponding loop $\lambda$ of $\cH$ on $\Sigma$ going from a nearby point $x$ on one side of $\lambda$ to a nearby point $y$ on the other. In the open connected components $\Omega_x$, $\Omega_y$ of $\Sigma \smallsetminus \cH$ adjacent to $\lambda$ that carry $x$ and $y$ respectively we take Jordan curves with odd winding numbers $\Gamma_x$ from $x$ back to $x$ and $\Gamma_y$ from $y$ back to $y$ and replace $\alpha$ by the combined arc $\tilde{\alpha} := \Gamma_x \alpha \Gamma_y$. Its lift $\tilde{\alpha}^*$ (the one that has the same endpoints as $\alpha^*$) then crosses the loop $\lambda^{\#\#}$ of the figure eight lift of $\lambda$ that has not been kept for the definition of $\cH^\#$ and so $\tilde{\alpha}^*$ is disjoint from $\cH^\#$. By replacing $\alpha^*$ with $\tilde{\alpha}^*$ in $A^*$ we have decreased the number of intersections with $\cH^\#$. Repeating the procedure we eventually get an arc $A^*$ from $p^*$ to $q^*$ that avoids $\cH^\#$ altogether. This finishes the proof in the `if' direction.

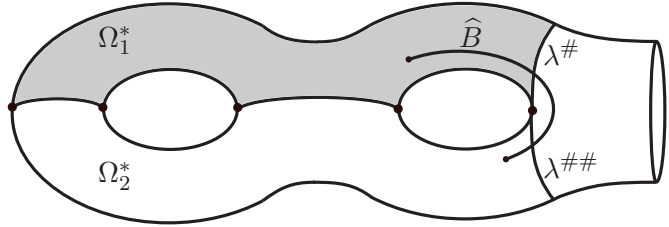
\begin{figure}[!ht]
\SetLabels
\L(.30*.80) $\Omega_1^*$\\
\L(.30*.20) $\Omega_2^*$\\
\L(.67*.72) $\lambda^{\#}$\\
\L(.67*.22) $\lambda^{\#\#}$\\
\L(.60*.80) $\widehat{B}$\\
\endSetLabels
\AffixLabels{%
\centerline{%
\includegraphics[height=3cm,width=8.8cm]{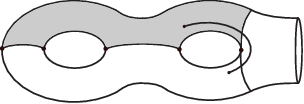}}}
\caption{A possible configuration of the two lifts $\Omega_1^*$ (in gray) and $\Omega_2^*$ in $S$.}
\label{fig:Lemma32_part2}
\end{figure}

For the  'only if' assume that $\Omega$ is a connected component of $\Sigma \smallsetminus \cH$ in which all Jordan curves have even winding numbers. By \eqref{eq:odd_winding2} its lift $\Omega^*$ in $S$ consists of two connected components $\Omega_1^*$, $\Omega_2^*$ (see \textit{Figure \ref{fig:Lemma32_part2}}). We claim that any arc $ \widehat{B}$ in $S$ with initial point in $\Omega_1^*$ and endpoint in $\Omega_2^*$ must intersect $\cH^\#$. For this we may assume that the projection $B : = \Pi \circ  \widehat{B}$ avoids the Weierstrass points and intersects $\cH$ only on its loops for otherwise we are done. By restricting $\widehat{B}$ to a sub arc, if necessary, we may further assume that $\widehat{B}$ leaves $\Omega_1^*$ only once and enters $\Omega_2^*$ only once. Let now $\lambda$ be the loop of $\cH$ on the boundary of $\Omega$ at which $B$ leaves the component $\Omega$ and let $\Lambda$ be the component on the other side of $\lambda$ into which $B$ enters at that moment. Since $\Sigma$ is a sphere (this is crucial here) the following hold,
\begin{flalign}\label{eq:sphereprop}
\begin{alignedat}{2}
&\rule{5pt}{0pt}& &\text{-- $\Lambda$ and $\Omega$ are distinct components,} \\[3pt]
&\rule{5pt}{0pt}& &\text{-- arc $B$ leaves and re-enters $\Omega$ at the same boundary loop $\lambda$.} 
\end{alignedat}&&
\end{flalign}
Let $\lambda^\#$, $\lambda^{\#\#}$ be the two loops of the figure eight lift $\lambda^*$ of $\lambda$. The first statement in \eqref{eq:sphereprop} implies that one of them, say $\lambda^\#$, is on the boundary of $\Omega_1^*$ and the other is on the boundary of $\Omega_2^*$. The second statement now implies that $\widehat{B}$ leaves $\Omega_1^*$ intersecting $\lambda^\#$ and enters $\Omega_2^*$ intersecting $\lambda^{\#\#}$. Thus, $\widehat{B}$ intersects $\cH^\#$. 

\end{proof}

The following lemma from surface topology shows that if the curves of $\cH$ on $\Sigma$ satisfy the condition from Lemma \ref{lem:lift_nonsep} then their lifts in $\cH^{\#}$  are  part of a basis of  $H_1(S,\Z)$.  In the following we call a \emph{partial basis} a set of homology classes that can be completed into a homology basis.

\begin{lem}
Let $S$ be a compact orientable surface of genus $g \geq 1$ and let $a_1, \dots, a_{n}$   be distinct curves on $S$ with the following properties
\begin{enumerate}
\item[1)] $a_1, \dots, a_{n}$ are simple closed curves,
\item[2)] $S \smallsetminus (a_1 \cup \dots \cup a_{n})$ is connected.
\end{enumerate}
Then $n \leq 2g$ and the homology classes $[a_1], \dots, [a_{n}]$ form a partial basis of $H_1(S,\Z)$.
\label{lem:rank_n}
\end{lem}

For convenience we sketch a proof. All arguments are standard.

\begin{proof} 

Without loss of generality we may assume that the system of curves is \emph{maximal}, i.e that it cannot be extended to a system $a_1, \dots, a_{n}, a_{n+1}$ satisfying 1) and 2). Then, cutting $S$ open along $a_1, \dots, a_{n}$ we obtain a topological sphere $S'$ with $m$ holes for some $m \geq 1$. In turn, $S$ is the quotient 

\begin{equation*}
S = S'/(\text{mod pasting})
\end{equation*}
``pasting'' meaning the reverse of the cutting. Note that any pair of points $p_1, p_2$ on the boundary $\partial S'$ of $S'$ that come together in this pasting must belong to the \emph{same} connected component of $\partial S'$ for otherwise we could draw a simple curve $a_{n + 1}$ from $p_1$ to $p_2$ on $S'$ and $a_1, \dots, a_{n}, a_{n+1}$ on $S$ would still satisfy 1) and 2) contradicting the maximality.

The argument is via graphs. For this the curves are homotoped so that they pairwise intersect at most finitely many times and $S$ is viewed as polyhedral surface on which $a_1, \dots, a_{n}$ \emph{are edge paths that pairwise intersect only in vertices}. Note that this may be carried out such that conditions 1) and 2) are still satisfied.

For any connected component $c$ of $\partial S'$  we let $F_c$ be a small tubular neighbourhood of $c$ in $S'$, ``small'' meaning that all $F_c$ are annuli and pairwise disjoint. For each $c$ we set 

\begin{equation*}
S_c = F_c/(\text{mod pasting}).
\end{equation*}
This represents $S$ as a direct sum of an $m$-holed sphere $S^0$ with $m$ surfaces $S_c, S_{c'}, \dots$ attached along the holes. Accordingly, the first homology group of $S$ is the direct sum of the first homology groups of $S_c, S_{c'}, \dots$. It suffices therefore to prove the lemma for each $S_c$ individually, and for this we may assume without loss of generality that $m=1$.   Let

\begin{equation*}
G = c/(\text{mod pasting})
\end{equation*}
be the connected graph obtained by restricting the pasting to $c$. As a point set  $G$ is the same as $a_1 \cup \dots \cup a_{n}$. Furthermore, $a_1, \dots, a_{n}$ are naturally identified as cycles of $G$. 

Since $G$ is a deformation retract of $S_c$ there is a natural isomorphism $\mathbf{j} : H_1(S_c,\Z) \to H_1(G,\Z)$ that acts as the identity on the homology classes of $a_1 , \dots,  a_{n}$. 

It remains to construct a basis of $H_1(G,\Z)$ via a spanning tree e.g. \cite[section 2.1.5]{st}. For this we delete from each cycle $a_k$  some edge $e_k$, $k=1,\dots,n$.  As the $a_k$ have no edges in common the $e_k$ are pairwise distinct and the resulting graph $G'$ is connected. If $G'$ still has cycles we delete successively  further edges $f_1, \dots f_q$ belonging to cycles $b_1, \dots, b_q$ until the resulting graph $G''$ is a tree. Then the homology classes of $a_1, \dots, a_{n}, b_1, \dots, b_q$   form a basis of $H_1(G,\Z)$. Via the isomorphism $\mathbf{j}$ it is identified with a basis of $H_1(S_c,\Z)$ and $n +q = 2g$. 
\end{proof}

Combining Lemmas \ref{lem:lift_nonsep} and \ref{lem:rank_n} we get

\begin{cor}
Assume that $\cH$ has altogether $n$ loops and edges and let $\alpha_1, \dots, \alpha_n$ be the corresponding simple closed curves of which $\cH^{\#}$ consists of. Then the homology classes $[\alpha_1], \dots, [\alpha_n]$ form a partial basis if and only if any open connected component of $\Sigma \smallsetminus \cH$ contains a simple closed curve $\Gamma$ that separates $p_1, \dots, p_{2g+2}$ into two odd subsets.
\label{cor:odd_curve}
\end{cor}

\section{A pruning algorithm} \label{sec:prune}
In this section we ``prune'' $\cG$ that is, we shall delete certain edges and loops so as to get homologically independent lifts in $S$. We recall that in this process the vertices shall be kept.

\begin{exa}
\emph{\textit{Figure \ref{fig:example_23}} shows on the first row a hypothetical graph $\cG$ consisting of $n$ blocks, $n$ an even number, where each block is a pair formed by a tree with one edge and a loop that surrounds it. The number of vertices is $2g+2 = 3n$.}
\label{exa:removeloops}
\end{exa}

\begin{figure}[!ht]
\vspace{-0.5cm}
\SetLabels
\L(.26*.42) $1$\\
\L(.43*.42) $2$\\
\L(.73*.42) $n$\\
\L(.86*.70) $\cG$\\
\L(.57*.70) $\ldots$\\
\L(.86*.20) $\cH$\\
\L(.57*.20) $\ldots$\\
\endSetLabels
\AffixLabels{%
\centerline{%
\includegraphics[width=10cm]{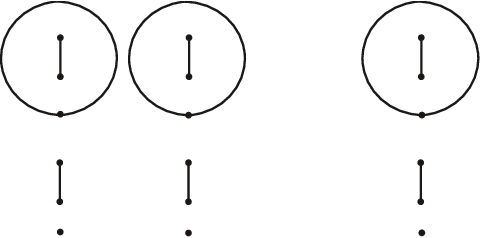}}}
\caption{Graph $\cG$ with $n$ loops. Deleting all loops yields subgraph $\cH$.}
\label{fig:example_23}
\end{figure}
For the hypothesis of Corollary \ref{cor:odd_curve} to hold we must delete for each block either the loop or the edge. Hence, the best  pruning for this example yields rank $n = \frac{1}{3}(2g+2)$.\\

We now show that a pruning with resulting rank of this order is always possible, where $\cG$ is again the graph as in Section \ref{sec:hom}.
\begin{lem} There exists a subgraph $\cH$ of $\cG$ for which $\cH^{\#}$ in $S$ consists of at least $\lceil \frac{2g+2}{3} \rceil$ closed curves that form a partial homology basis.
\label{lem:2g3pruning}
\end{lem} 
\begin{proof} In two preliminary steps we simplify $\cG$ and then continue by an algorithm.\\

\textbf{Preliminary step 1}:\; From each connected component $\cT$ of $\cG$ which is a looped tree (cf. Lemma \ref{lem:graph_nooroneloop}) we delete the loop.\\

For instance, if $\cT$ is as the fourth example in \textit{Figure \ref{fig:graphcomponents}} then the deletion of the loop yields the third example in that figure. This preliminary step  leaves us with a subgraph $\cG'$ of $\cG$ with $2g+2$ vertices having the property that \emph{each connected component of $\cG'$ is either a tree or a loop}. In particular, $\cG'$ has no isolated vertices.\\

\textbf{Preliminary step 2}:\; For each connected component $\cT$ of $\cG'$ that is a tree with $k\geq 2$ edges we delete an edge leading to a leaf.\\

Furthermore,  the remaining subtree $\cT'$ of $\cT$ with $k-1$ edges and the leaf $p$ that has become isolated are grouped into a pair $\{\cT',p\}$. The pair shall be called a \emph{paired block}.\\

Let $\cH$ be the subgraph of $\cG'$ resulting from preliminary step 2. Any connected component of $\cH$ that is not part of a paired block shall be called a \emph{singleton block}. It follows from the construction that a singleton block can only be either a loop or a tree with exactly one edge. Components of the latter type shall be called \emph{bones}. In the example of \textit{Figure \ref{fig:levels}}, for instance, there are three bones and three paired blocks. These are marked by dotted lines that surround them; these lines do not belong to the graph; the bone in the region of level 4 belongs to a paired block, the other two bones don't.

Arranging components into such blocks shall allow us to keep track of the number of edges and loops that remain in the pruning process.\\

The further pruning is now carried out in an algorithmic way. For this we let $L$ be the union of all loops of $\cH$. For the ease of exposition we assume that $L$ is not empty, otherwise the following algorithm has just one step.

We shall call  \emph{regions} the open connected components of $\Sigma \smallsetminus L$. We order them hierarchically into levels as follows. Since $\Sigma$ is a sphere, at least two regions are disks. We let $Q$ with boundary loop $\lambda_Q$ be one of them and set it to be at the highest or, speaking with \textit{Figure \ref{fig:levels}}, the \emph{outermost} level.

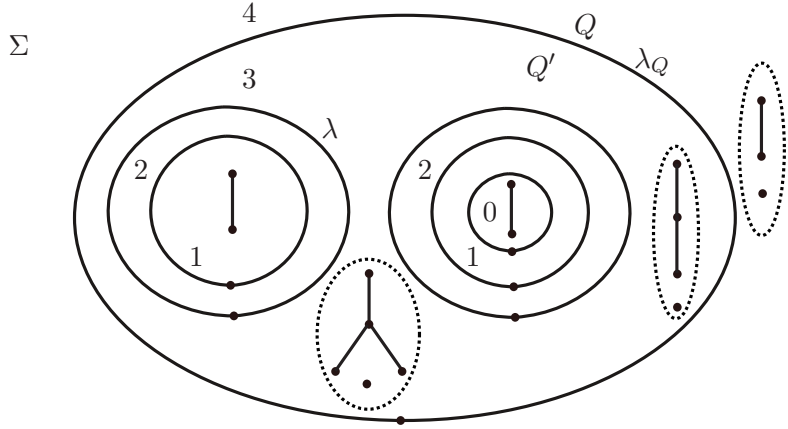
\begin{figure}[!ht]
\SetLabels
\L(.15*.90) $\Sigma$\\
\L(.62*.95) $Q$\\
\L(.67*.87) $\lambda_Q$\\
\L(.58*.85) $Q'$\\
\L(.41*.70) $\lambda$\\
\L(.34*.98) $\,4$\\
\L(.34*.82) $\,3$\\
\L(.25*.60) $\,2$\\
\L(.49*.60) $2$\\
\L(.30*.39) $1$\\
\L(.53*.39) $1$\\
\L(.54*.50) $\,0$\\
\endSetLabels
\AffixLabels{%
\centerline{%
\includegraphics[height=5.5cm]{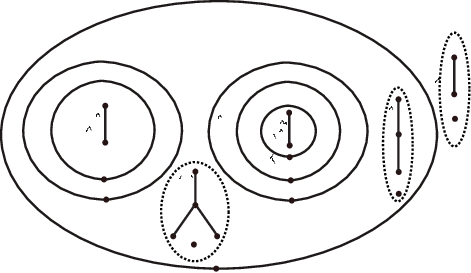}}}
\caption{Regions of level $0,1,2,3$ and $4$ on the sphere $\Sigma$. In the region of level 3 there are two paired blocks. This is indicated by dotted lines.}
\label{fig:levels}
\end{figure}

The second highest level is the region $Q'$ adjacent to $Q$ along $\lambda_Q$. We call $\lambda_Q$ its \emph{outer boundary}  and all other loops on the boundary of $Q'$ are called the \emph{inner boundary loops} (supposing for the ease of exposition that there are some, otherwise $Q'$ is already at the lowest level). For any inner boundary loop $\lambda$ of $Q'$ the region adjacent to $Q'$ along $\lambda$ is at the third highest level with $\lambda$ its outer boundary loop and again the remaining boundary loops being called the inner ones. In this way we proceed. In the end each region is at some level and, except for the highest and the lowest level, for any region of level $k$ there is one region of level $k+1$ adjacent to it along the outer boundary and there are regions of level $k-1$ adjacent along all inner boundary loops. We enumerate the levels such that the lowest level is 0. \textit{Figure \ref{fig:levels}} shows a case with levels going from 0 to 4.

The pruning algorithm now runs level wise starting with level 0. Every time an edge or a loop is deleted the resulting subgraph will be renamed $\cH$ again. When a loop $\lambda$ is deleted the region of the lower level adjacent to $\lambda$ together with $\lambda$ itself shall be merged to the adjacent region of the higher level and the new region thus formed assumes the higher of the two level numberings. Hence, the level numbering is not re-adjusted and some levels may eventually become void. During the algorithm some singleton blocks disappear and new paired blocks $\{\cT, p\}$ are formed, where $p$ is an isolated vertex and for $\cT$ we allow that
\begin{enumerate}
\item[--] $\cT$ is either a bone in the same region as $p$ or one of the inner boundary loops of the region that contains $p$.
\end{enumerate}
Paired blocks once created shall not be altered in subsequent steps. The result of the algorithm will be that in the end each region contains an isolated vertex. \\

\textbf{Algorithm:} Let now $k \geq 0$ be a level, $\cH$ the current version of the subgraph of $\cG'$ at the beginning of the new step of the algorithm and assume the \emph{induction hypothesis} that
\begin{enumerate}
\item[--] any region of level $\leq k-1$ contains an isolated vertex,
\item[--] $\cH$ is the disjoint union of singleton blocks and paired blocks,
\end{enumerate}
where we recall that each singleton block is either a bone or a loop. Let $\Omega$ be the region of level $k$ the algorithm is going to treat in the current step. By definition it acts according to the following nested  `IF-ELSE' instructions accompanied by short comments. Case 2. and 3. and Case 4. and 5. are illustrated in \textit{Figure \ref{fig:pruning_case23}} and \textit{Figure \ref{fig:pruning_case45}}, respectively. 

\begin{enumerate}
\item {\bf $\Omega$ contains an isolated vertex:} no action is needed.
\item {\bf $\Omega$ has at least two inner boundary loops:} delete one of them, merge the corresponding level $k-1$ region with $\Omega$ and create a paired block consisting of the vertex that has become isolated and one of the remaining inner boundary loops of $\Omega$. Observe that the latter does not yet belong to a paired block formed earlier, hence this operation is admissible.
\item {\bf $\Omega$ has exactly one inner boundary loop and contains at least one bone: } proceed as in 2.\ but take one of the bones in $\Omega$ to form the paired block. Note that since $\Omega$ contains no isolated vertex it contains no paired block with a bone. Hence the bone does not belong to a paired block formed earlier and the operation is admissible.

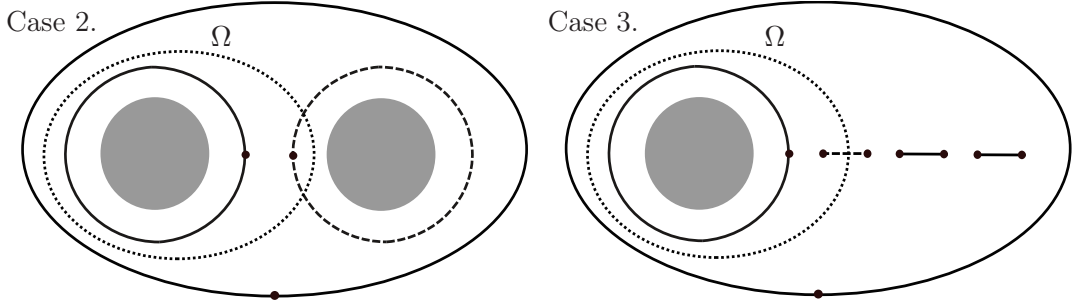
\begin{figure}[!ht]
\SetLabels
\L(.05*.90) Case 2.\\
\L(.22*.85) $\Omega$\\
\L(.50*.90) Case 3.\\
\L(.68*.85) $\Omega$\\
\endSetLabels
\AffixLabels{%
\centerline{%
\includegraphics[height=4cm,width=14cm]{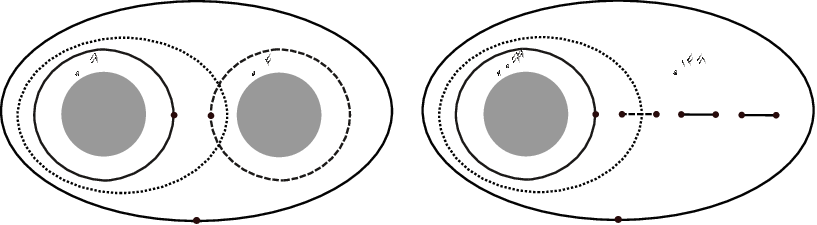}}}
\caption{Pruning: Cases 2 and 3: Paired blocks are indicated by dotted surrounding lines that are not part of the graph, deleted edges or loops are drawn as dashed lines. Gray fields indicate content of loops.}
\label{fig:pruning_case23}
\end{figure}

\item {\bf $\Omega$ has exactly one inner boundary loop but contains no bones: }  delete the outer boundary loop of $\Omega$, merge $\Omega$ with the adjacent $k+1$ level region and create a paired block using the inner boundary component and the freed vertex. (Since $\Omega$ contains neither bones nor paired blocks it must be either a disk or an annulus. But a disk is excluded by hyperbolic geometry, hence the existence of an outer boundary.)

\begin{figure}[!ht]
\SetLabels
\L(.05*.90) Case 4.\\
\L(.17*.60) $\Omega$\\
\L(.50*.90) Case 5.\\
\L(.63*.60) $\Omega$\\
\endSetLabels
\AffixLabels{%
\centerline{%
\includegraphics[height=4cm,width=14cm]{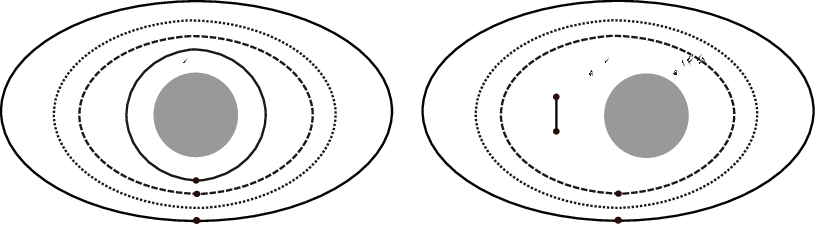}}}
\caption{Pruning: Cases 4 and 5}
\label{fig:pruning_case45}
\end{figure}
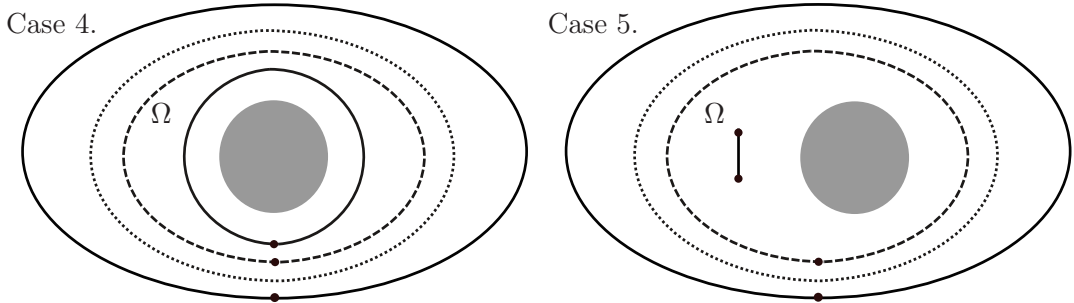

\item {\bf $\Omega$ has no inner boundary loop but an outer boundary loop:} proceed as in 4.\ but take one of the bones in $\Omega$ to form a paired block. (Since $\Omega$ contains no paired blocks and cannot be isometric to a hyperbolic disk it must contain cone points and thus bones.)

\item $\pmb{\Omega = \Sigma}$: This is the remaining case. By what is ruled out through the preceding cases, $\cH$ consists of $g+1$ bones and the instruction is to delete an edge and form two paired blocks,  possible because $g+1 \geq 3$.
\end{enumerate}

The algorithm stops after the highest level is treated. At that moment any connected component of $\Sigma \smallsetminus \cH$ contains an isolated vertex (each being part of a paired block). Hence, the hypothesis of Corollary \ref{cor:odd_curve} is satisfied and $\cH^{\#}$ consists of homologically independent closed curves. Furthermore, $\cH$ is the disjoint union of paired blocks and singleton blocks. In each block the number of edges or loops is at least $\frac{1}{3}$ times the number of vertices. Hence, $\cH$ has at least $\frac{1}{3}(2g+2)$ edges and loops. The lemma now follows from Corollary \ref{cor:odd_curve}. 
\end{proof}

We conclude, re-enumerating the geodesics, for simplicity:

\begin{thm} Let $S$ be a hyperelliptic Riemann surface of genus $g \geq 2$. Then there exist $\lceil \frac{2g+2}{3}  \rceil$ geodesic loops  $(\alpha_k)_{k=1,\ldots,\left \lceil \frac{2g+2}{3}\right \rceil}$, forming a partial homology basis such that
\begin{equation*}
\ell(\alpha_{k}) \leq  4\log \left(\frac{ 12(g-1) }{ 2g+5-3k} +2 \right) \text{ for all } k=1, \dots, \left\lceil \frac{2g+2}{3}\right \rceil.
\end{equation*}
\label{thm:hyper_lengths}
\end{thm}

Summarizing this result for a fraction of the homology basis, we have:

\begin{cor} Let $S$ be a hyperelliptic Riemann surface of genus $g \geq 2$. Then for any $\lambda \in (0,1)$ there exist $\lceil \lambda \cdot \frac{2}{3} g \rceil$ geodesic loops  $(\alpha_k)_{k=1,\ldots,\left \lceil \lambda \cdot \frac{2}{3}g \right \rceil}$, forming a partial homology basis such that
\begin{equation}\label{eq:hyper_lengths_cor}
\ell(\alpha_k) \leq N(\lambda) :=  4\log \left( \frac{ 6}{ 1-\lambda} +2 \right)   \text{ for all } k \in \{1 ,\ldots, \left \lceil   \lambda \cdot \frac{2}{3}g \right \rceil \}.
\end{equation}
\label{thm:hyper_lengths_cor}
\end{cor}
  
\begin{proof}[Proof of Theorem \ref{thm:hyper_lengths}] 

With the procedure in Section \ref{sec:loops} we found simple closed geodesics $\alpha_m$, $m=1,\dots,M$, on $S$ with lengths bounded above by Lemma \ref{thm:alpham}, to be recalled below, and by Lemma \ref{lem:2g3pruning} there exists a selection

\begin{equation*}
\alpha_{m_1}, \alpha_{m_2}, \dots, \alpha_{m_{\kappa(g)}}, \quad \kappa(g) := \left\lceil \frac{2g+2}{3}\right\rceil
\end{equation*}
from these that forms a partial basis. We take the enumeration such that $m_1 < m_2 < \dots < m_{\kappa(g)}$. We then conclude with Lemma \ref{thm:alpham} that

\begin{equation*}
\ell(\alpha_{m_k}) \leq 4\log \left(\frac{ 4(g-1) }{ 2g+2-j_{m_k}} +2 \right),
\end{equation*}
where the $j_{m_k}$ are monotonically increasing and $j_{m_{\kappa(g)}} \leq 2g+2$. From this it follows that $j_{m_k}\leq 2g+2 - \kappa(g) + k$ and, hence

\begin{equation*}
\ell(\alpha_{m_k}) \leq 4\log \left(\frac{ 4(g-1) }{ \kappa(g)+1-k} +2 \right) \leq 4\log \left(\frac{ 12(g-1) }{ 2g+5-3k} +2 \right),\quad k=1, \dots, \left\lceil \frac{2g+2}{3}\right \rceil.
\end{equation*}
In our statement the curves $\alpha_{m_k}$ are renamed $\alpha_k$. This proves Theorem \ref{thm:hyper_lengths}.  
\end{proof}

Taking $\lambda \in (0,1)$ and restricting $k$ to $k \leq \left\lceil \lambda \cdot \frac{2g}{3}\right \rceil < \lambda \cdot \frac{2g+2}{3}+1$ we get the statement of Corollary \ref{thm:hyper_lengths_cor}, where $2g+2$ is simplified to $2g$.

\section{Jacobian of hyperelliptic surfaces} \label{sec:Jac_hyp}

In this part we apply our results to the study of the Jacobian torus $J(S)$ of a Riemann surface $S$. The following approach to the Jacobian can also be found in \cite{bs} or \cite{bps}. In what follows a differential form or 1-form on a Riemann surface $S$ is understood to be a \emph{real} differential 1-form. The \textit{energy} of a 1-forms $\omega$ of class $L^2$ on $S$ is given by

\begin{equation*}\label{eq:scalpr}
        E(\omega) =  \int_S \omega \wedge \star \omega,
\end{equation*}
where $ \wedge $ denotes the wedge product and $\star$ the Hodge star operator. When $S$ is compact, then in each cohomology class of closed 1-forms on $S$ there is a unique harmonic form and the latter is an energy minimizer, i.e.\  the harmonic form is the unique element in its cohomology class that has minimal energy.

The tool here is the approximation of the energy of a harmonic form in a cylinder. By the Collar Lemma for hyperbolic surfaces, from \cite[p.\ 106]{bu},  any simple closed geodesic $\gamma$ on $S$ is embedded in a cylindrical neighbourhood $C(\gamma)$, whose width $w_s$ is given by its length $\ell(\gamma)$:
\begin{equation}
 w_s =  \arcsinh \left( \frac{1}{\sinh \left(\ell(\gamma) / 2\right)} \right).
\label{eq:col_lemma2}
\end{equation}
Let $(\alpha_k)_{k=1,\ldots,2g}$ be a set of simple closed geodesics on $S$ that induce a basis of $H_1(S,\Z)$. In \cite{bs} an upper bound for the energy $E(\sigma_k)$ of a harmonic 1-form $\sigma_k \in H^1(S,\Z)$ is provided:
\begin{equation}
     E(\sigma_k) \leq \frac{\ell(\alpha_k)}{\pi-2 \arcsin \left(\frac{1}{\cosh(w)} \right)}
\label{eq:capa_omega}
\end{equation}
Here $w$ denotes the width of the cylinder $C$ around $\alpha_k$, where $w$ can be different, especially larger than, $w_s$. These $(\sigma_k)_{k=1,\ldots,2g}$ form a basis of $H^1(S,\Z)$. Furthermore $E(\sigma_k)$  is equal to the squared length $\|v_k\|^2$ of a lattice vector $v_k$ in $J(S)$. Again, these $(v_k)_{k=1,\ldots,2g}$ form a basis of the lattice. 
Let $S$ be a hyperelliptic hyperbolic surface of genus $g \geq 2$. By Corollary \ref{thm:hyper_lengths_cor} for any $\lambda \in (0,1)$ there exist $\lceil \lambda \cdot \frac{2}{3} g \rceil$ homologically independent loops $(\alpha_k)_{k=1,\ldots,\left \lceil \lambda \frac{2}{3}g \right \rceil}$, such that $\ell(\alpha_k) \leq N(\lambda)$, where $N(\lambda)$ is from \eqref{eq:hyper_lengths_cor}.

Using this result in \eqref{eq:col_lemma2}, we obtain an upper bound for the energy in \eqref{eq:capa_omega}, which also depends on $\lambda$. To this end we set
\[
           w(\lambda) =  \arcsinh \left( \frac{1}{\sinh(\frac{N(\lambda)}{2})} \right).
\]
In total we obtain the following corollary, which is Theorem \ref{thm:hyper_Jac_intro} of the introduction:

\begin{cor} Let $S$ be a hyperelliptic hyperbolic surface of genus $g \geq 2$. Then for any $\lambda \in (0,1)$ there exist $\lceil \lambda \cdot \frac{2}{3} g \rceil$ linearly independent vectors $(v_k)_{k=1,\ldots,\left \lceil \lambda \frac{2}{3}g \right \rceil}$  in the lattice of the Jacobian torus $J(S)$ that can be extended to a lattice basis, such that
\[
  \|v_k\|^2 = E(\sigma_k) \leq D(\lambda): = \frac{N(\lambda)}{\pi-2 \cdot \arcsin(\frac{1}{\cosh(w(\lambda)})}   \text{ \ for all \ } k \in \{1 ,\ldots, \left \lceil \lambda \frac{2}{3}g \right \rceil \}.
\]
\label{thm:hyper_cor2}
\end{cor}
Note that the the upper bound $D(\lambda)$ is a constant that does not depend on the genus.

\section*{Acknowledgment}

The authors would like to thank the referee for his/her careful reading of the manuscript and the very helpful comments which have greatly improved this paper.

\newpage

\vspace{1cm}

\noindent Peter Buser \\
\noindent Department of Mathematics, Ecole Polytechnique F\'ed\'erale de Lausanne\\
\noindent Station 8, 1015 Lausanne, Switzerland\\
\noindent e-mail: \textit{peter.buser@epfl.ch}\\
\\
\\
\noindent Eran Makover\\
\noindent Department of Mathematics, Central Connecticut State University\\
\noindent 1615 Stanley Street, New Britain, CT 06050, USA\\
\noindent e-mail: \textit{makovere@ccsu.edu}\\
\\
\\
\noindent Bjoern Muetzel \\
\noindent Department of Mathematics, Eckerd College \\
\noindent 4200 54th avenue South, St. Petersburg, FL 33711, USA\\
\noindent e-mail: \textit{bjorn.mutzel@gmail.com}\\


\begin{thebibliography} {[W/M/W]}
\normalsize

\bibitem[Ah]{ah} Ahlfors, L.V. : \textit{Complex Analysis}, third edition, McGraw-Hill, (1979)

\bibitem[Ba]{ba2}
Bavard, C. : \textit{La systole des surfaces hyperelliptiques}, pr{\'e}publication de l'ENS Lyon, {\bf71} (1992).

\bibitem[Be]{be}
Bers, L. : \textit{An inequality for Riemann surfaces}, Differential Geometry and Complex Analysis, I. Chavel and H. Farkas Eds., Springer Verlag, (1985), 87--93.

\bibitem[BP]{bp}
Balacheff F., Parlier, H. : \textit{Bers' constants for punctured spheres and hyperelliptic surfaces}, J. Topol. Anal. 4 (2012), 271--296. 

\bibitem[BPS]{bps}
Balacheff, F., Parlier, H. and Sabourau, S. : \textit{Short loop decompositions of surfaces and the geometry of Jacobians}, Geom. Funct. Anal. {\bf22}(1) (2012), 37--73.

\bibitem[BS]{bs}
Buser, P. and Sarnak, P. : \textit{On the Period Matrix of a Riemann Surface of Large Genus (with an Appendix by Conway,J.H. And Sloane,N.J.A.)}, Inventiones Mathematicae {\bf117}(1) (1994), 27--56.

\bibitem[BSe]{bse0}
Buser, P. and Sepp\"al\"a, M. : \textit{Symmetric pants decompositions of Riemann surfaces}, Duke Math. J. {\bf67}(1) (1992),
39--55. 



\bibitem[Bu]{bu}
Buser, P. : \textit{ Geometry and Spectra of compact Riemann surfaces}, Progress in Mathematics (106), Birkh\"auser Verlag, Boston, (1992).



\bibitem[Hw]{hw} Hwang, JM. : \textit{Buser-Sarnak invariants of Prym varieties}, 
The Michigan Mathematical Journal {\bf 62}(3) (2013), 665--671. 

\bibitem[Mu]{mu1}
Muetzel, B. : \textit{On the second successive minimum of the Jacobian of a Riemann surface}, Geom. Dedicata {\bf161}(1) (2012), 85--107.



\bibitem[Pa]{pa1}
Parlier, H. : \textit{A short note on short pants}, Can. Math. Bull. {\bf57}(4) (2014), 870--876.

\bibitem[St]{st}
Stillwell, J. : \textit{Classical Topology and Combinatorial Group Theory}, second edition, Springer, (1993).

\bibliographystyle{plain}
\end{thebibliography}
\end{document}